\documentclass[reqno]{amsart}
\usepackage{amssymb}
\usepackage{graphicx}

\usepackage[usenames, dvipsnames]{color}
\usepackage{verbatim}
\usepackage{mathrsfs}
\usepackage{bm}
\usepackage{cite}

\numberwithin{equation}{section}

\newtheorem{theorem}{Theorem}[section]

\newtheorem{lemma}[theorem]{Lemma}

\theoremstyle{definition}
\newtheorem{remark}[theorem]{Remark}

\theoremstyle{definition}

\theoremstyle{definition}

\makeatletter
\def\dashint{\operatorname%
{\,\,\text{\bf-}\kern-.98em\DOTSI\intop\ilimits@\!\!}}
\makeatother

\def\\det{\text{det}}

\def\.5{\frac{1}{2}}

\newcommand{\RN}[1]{%
  \textup{\uppercase\expandafter{\romannumeral#1}}%
}

\renewcommand{\epsilon}{\varepsilon}

\newcounter{marnote}

\begin{document}
\title[Asymptotics for the perfect conductivity problem ]{Asymptotics for the electric field when $M$-convex inclusions are close to the matrix boundary}

\author[Z.W. Zhao]{Zhiwen Zhao}

\address[Z.W. Zhao]{1. School of Mathematical Sciences, Beijing Normal University, Beijing 100875, China. }
\address{2. Bernoulli Institute for Mathematics, Computer Science and Artificial Intelligence, University of Groningen, PO Box 407, 9700 AK Groningen, The Netherlands.}
\email{zwzhao@mail.bnu.edu.cn.}


\date{\today} 



\begin{abstract}
In the perfect conductivity problem of composites, the electric field may become arbitrarily large as $\varepsilon$, the distance between the inclusions and the matrix boundary, tends to zero. The main contribution of this paper lies in developing a clear and concise procedure to establish a boundary asymptotic formula of the concentration for perfect conductors with arbitrary shape in all dimensions, which explicitly exhibits the singularities of the blow-up factor $Q[\varphi]$ introduced in \cite{LX2017} by picking the boundary data $\varphi$ of $k$-order growth. In particular, the smoothness of inclusions required for at least $C^{3,1}$ in \cite{LLY2019} is weakened to $C^{2,\alpha}$, $0<\alpha<1$ here.
\end{abstract}

\maketitle



\section{Introduction and main results}

It is well known that field concentrations appear widely in nature and industrial applications. These fields include extreme electric, heat fluxes and mechanical loads. Motivated by the issue of material failure initiation, in this paper we are devoted to the investigation of blow-up phenomena arising from high-contrast fiber-reinforced composites with the densely packed fibers. The key feature of the concentrated fields is that the blow-up comes from the narrow regions between fibers and the thin gaps between fibers and the matrix boundary. It is worth emphasizing that the latter is more interesting due to the interaction from the boundary data. Although there has made great progress in the engineering and mathematical literature since Babu\u{s}ka et al's famous work \cite{BASL1999} over the past two decades, accurate numerical computation of the concentrated field are still very hard for lack of fine characterization to develop an efficient numerical scheme. So, it is significantly important from a practical point of view to precisely describe the singular behavior of such high concentration.

In the context of electrostatics, the field is the gradient of a solution to the Laplace equation and the blow-up rate of the gradient were captured accurately. Denote the distance between two inclusions or between inclusions and the matrix boundary by $\varepsilon$. It has been proved that for the perfect conductivity problem, the blow-up rate of the gradient is $\varepsilon^{-1/2}$ in two dimensions \cite{AKLLL2007,BC1984,BLY2009,AKL2005,Y2007,Y2009,K1993}, while it is $|\varepsilon\ln\varepsilon|^{-1}$ in three dimensions \cite{BLY2009,LY2009,BLY2010,L2012}.

Besides these foregoing estimates of the singularities for the field, there is another direction of investigation to establish the asymptotic formula of $\nabla u$ in the thin gap of electric field concentration. In two dimensions, consider the following conductivity problem
\begin{align}\label{per001}
\begin{cases}
\Delta u=0,&\hbox{in}\;\mathbb{R}^{2}\setminus\overline{D_{1}\cup D_{2}},\\
u=C_{j}, &\hbox{on}\;\partial D_{j},\;j=1,2,\\
u(\mathbf{x})-H(\mathbf{x})=O(|\mathbf{x}|^{-1}),&\mathrm{as}\;|\mathbf{x}|\rightarrow\infty,\\
\int_{\partial D_{j}}\frac{\partial u}{\partial\nu}\big|_{+}=0,&j=1,2,
\end{cases}
\end{align}
where $H$ is a given harmonic function in $\mathbb{R}^{2}$ and
$$\frac{\partial u}{\partial\nu}\Big|_{+}:=\lim_{\tau\rightarrow0}\frac{u(x+\nu\tau)-u(x)}{\tau}.$$
Here and throughtout this paper $\nu$ is the unit outer normal of $D_{j}$ and the subscript $\pm$ shows the limit from outside and inside the domain, respectively. For problem (\ref{per001}), Kang, Lim and Yun \cite{KLY2013} obtained a complete characterization of the singularities of $\nabla u$ with $D_{1}$ and $D_{2}$ being disks as follows
\begin{align}\label{singular}
\nabla u(\mathbf{x})=\frac{2r_{1}r_{2}}{r_{1}+r_{2}}(\mathbf{n}\cdot\nabla H)(\mathbf{p})\nabla h(\mathbf{x})+\nabla g(\mathbf{x}),
\end{align}
where $h(\mathbf{x})=\frac{1}{2\pi}(\ln|\mathbf{x}-\mathbf{p}_{1}|-\ln|\mathbf{x}-\mathbf{p}_{2}|)$ with $\mathbf{p}_{1}\in D_{1}$ and $\mathbf{p}_{2}\in D_{2}$ being the fixed point of $R_{1}R_{2}$ and $R_{2}R_{1}$ respectively, $R_{j}$ is the reflection with respect to $\partial D_{j}$, $\mathbf{n}$ is the unit vector in the direction of $\mathbf{p}_{2}-\mathbf{p}_{1}$, $\mathbf{p}$ is the middle point of the shortest line segment connecting $\partial D_{1}$ and $\partial D_{2}$, and $|\nabla g|$ is bounded independently of $\varepsilon$ on any bounded subset of $\mathbb{R}^{2}\setminus\overline{D_{1}\cup D_{2}}$. Obviously $\nabla h$ characterizes the singular behavior of $\nabla u$ explicitly. Ammari, Ciraolo, Kang, Lee, Yun \cite{ACKLY2013} extended the characterization (\ref{singular}) to the case when inclusions $D_{1}$ and $D_{2}$ are strictly convex domains in $\mathbb{R}^{2}$ by utilizing disks osculating to convex domains. In three dimensions, Kang, Lim and Yun \cite{KLY2014} derived an asymptotic formula of $\nabla u$ for two spherical perfect conductors with the same radii. The asymptotics for perfectly conducting particles with the different radii can be seen in \cite{LWX2019}. Recently, a great work on establishing an asymptotic formula in dimensions two and three for two arbitrarily $2$-convex inclusions has been completed by Li, Li and Yang in \cite{LLY2019}. It is worth mentioning that for high-contrast composites with the matrix described by nonlinear constitutive laws such as $p$-Laplace, Gorb and Novikov \cite{G2012} captured the stress concentration factor. Additionally, the asymptotics of the eigenvalues of the Poincar\'{e} variational problem for two close-to-touching inclusions were obtained by Bonnetier and Triki in \cite{BT2013}. More related work can be seen in \cite{ABTV2015,AKLLZ2009,BCN2013,BLL2015,BLL2017,BT2012,BV2000,DL2019,G2015,KY2019,KLY2015,LLBY2014,LY2015,M1996,MMN2007,BJL2017,LX2017}.

However, to the best of our knowledge, previous investigations on the asymptotics of the field concentration only focused on the narrow region between inclusions. This paper, by contrast, aims at deriving a completely asymptotic characterization for the perfect conductivity problem (\ref{con002}) with $m$-convex inclusions close to the matrix boundary and the boundary data of $k$-order growth in all dimensions. The asymptotic results in this paper also provide an efficient way to compute the electrical field numerically.

To state our main works in a precise manner, we first describe our domain and notations. Let $D\subset\mathbb{R}^{n}\,(n\geq2)$ be a bounded domain with $C^{2,\alpha}~(0<\alpha<1)$ boundary, which has a $C^{2,\alpha}$-subdomain $D_{1}^{\ast}$ touching matrix boundary $\partial D$ only at one point. That is, by a translation and rotation of the coordinates, if necessary,
\begin{align*}
\partial D_{1}^{\ast}\cap\partial D=\{0'\}\subset\mathbb{R}^{n-1}.
\end{align*}
Throughout the paper, we use superscript prime to denote ($n-1$)-dimensional domains and variables, such as $\Sigma'$ and $x'$. After a translation, we set
\begin{align*}
D_{1}^{\varepsilon}:=D_{1}^{\ast}+(0',\varepsilon),
\end{align*}
where $\varepsilon>0$ is a sufficiently small constant. For the sake of simplicity, denote
\begin{align*}
D_{1}:=D_{1}^{\varepsilon},\quad\mathrm{and}\quad\Omega:=D\setminus\overline{D}_{1}.
\end{align*}

The conductivity problem with inclusions close to touching matrix boundary can be modeled by the following scalar equation with piecewise constant coefficients
\begin{align}\label{con001}
\begin{cases}
\mathrm{div}{(a_{k}(x)\nabla u)}=0,&\hbox{in}\;D,\\
u=\varphi, &\hbox{on}\;\partial{D},
\end{cases}
\end{align}
where
\begin{align*}
a_{k}(x)=
\begin{cases}
k\in[0,1)\cup(1,\infty],&\hbox{in}\;D_{1},\\
1,&\hbox{on}\;D\setminus D_{1}.
\end{cases}
\end{align*}
Actually, equation (\ref{con001}) can also be used to describe more physical phenomenon, such as dielectrics, magnetism, thermal conduction, chemical diffusion and flow in porous media.

When the conductivity of $D_{1}$ degenerates to be infinity, problem (\ref{con001}) turns into the perfect conductivity problem as follows
\begin{align}\label{con002}
\begin{cases}
\Delta u=0,&\hbox{in}\;D\setminus D_{1},\\
u=C_{1}, &\hbox{in}\;\overline{D}_{1},\\
\int_{\partial D_{1}}\frac{\partial u}{\partial\nu}\big|_{+}=0,\\
u=\varphi, &\mathrm{on}\;\partial D,
\end{cases}
\end{align}
where the free constant $C_{1}$ is determined later by the third line of (\ref{con002}). There has established the existence, uniqueness and regularity of weak solutions to (\ref{con002}) in \cite{BLY2009} with a minor modification. We further assume that there exists a small constant $R>0$ independent of $\varepsilon$, such that the portions of $\partial D$ and $\partial D_{1}$ near the origin can be written as
\begin{align*}
x_{n}=\varepsilon+h_{1}(x')\quad\mathrm{and}\quad x_{n}=h(x'),\quad\quad x'\in B_{2R}',
\end{align*}
where $h_{1}$ and $h$ satisfy that for $m\geq2$,
\begin{enumerate}
{\it\item[(\bf{\em H1})]
$h_{1}(x')-h(x')=\lambda|x'|^{m}+O(|x'|^{m+1}),$
\item[(\bf{\em H2})]
$|\nabla_{x'}^{i}h_{1}(x')|,\,|\nabla_{x'}^{i}h(x')|\leq \kappa_{1}|x'|^{m-i},\;\,i=1,2,$
\item[(\bf{\em H3})]
$\|h_{1}\|_{C^{2,\alpha}(B'_{2R})}+\|h\|_{C^{2,\alpha}(B'_{2R})}\leq \kappa_{2},$}
\end{enumerate}
where $\lambda$ and $\kappa_{j},j=1,2$, are three positive constants independent of $\varepsilon$.

To explicitly uncover the effect of boundary data $\varphi$ on the singularities of the field, we classify $\varphi\in C^{2}(\partial D)$ according to its parity as follows. Denote the bottom boundary of $\Omega_{R}$ by $\Gamma^{-}_{R}=\{x\in\mathbb{R}^{n}|\,x_{n}=h(x'),\,|x'|<R\}$. Suppose that for $x\in\Gamma^{-}_{R}$,
\begin{itemize}
{\it
\item[({\bf{\em S1}})] $\varphi$ satisfies the $k$-order growth condition, that is,
\begin{align*}
\varphi(x)=\eta\,|x'|^{k};
\end{align*}
\item[({\bf{\em S2}})] $\varphi$ is odd with respect to some $x_{i_{0}}$, $i_{0}\in\{1,\cdots,n-1\}$,}
\end{itemize}
where $\eta>0$ and $k>1$ is a positive integer.

For $z'\in B'_{R},\,0<t\leq2R$, denote
\begin{align*}
\Omega_{t}(z'):=&\left\{x\in \mathbb{R}^{n}~\big|~h(x')<x_{n}<\varepsilon+h_{1}(x'),~|x'-z'|<{t}\right\}.
\end{align*}
We will use the abbreviated notation $\Omega_{t}$ for the domain $\Omega_{t}(0')$. Before stating our main results, we first introduce two scalar auxiliary functions $\bar{u}\in C^{2}(\mathbb{R}^{n})$ and  $\bar{u}_{0}\in C^{2}(\mathbb{R}^{n})$ such that $\bar{u}=1$ on $\partial D_{1}$, $\bar{u}=0$ on $\partial D$ and
\begin{align}\label{con009}
\bar{u}(x)=\frac{x_{n}-h(x')}{\varepsilon+h_{1}(x')-h(x')},\;\,\mathrm{in}\;\,\Omega_{2R},\quad\;\|\bar{u}\|_{C^{2}(\Omega\setminus\Omega_{R})}\leq C,
\end{align}
and $\bar{u}_{0}=0$ on $\partial D_{1}$, $\bar{u}_{0}=\varphi(x)$ on $\partial D$, and
\begin{align}\label{con016}
\bar{u}_{0}=\varphi(x',h(x'))(1-\bar{u}),\;\,\mathrm{in}\;\Omega_{2R},\quad\;\|\bar{u}_{0}\|_{C^{2}(\Omega\setminus\Omega_{R})}\leq C.
\end{align}

To simplify notations used in the following, for $i=0$, and $i=k$, $k$ is the order of growth defined in ({\bf{\em S1}}), we denote
\begin{align*}
\rho_{i}(n,m;\varepsilon)=&
\begin{cases}
\varepsilon^{\frac{n+i-1}{m}-1},&m>n+i-1,\\
|\ln\varepsilon|,&m=n+i-1,\\
1,&m<n+i-1,
\end{cases}
\end{align*}
and
\begin{align*}
\Gamma_{m}^{n+i}=&
\begin{cases}
\Gamma\left(1-\frac{n+i-1}{m}\right)\Gamma\left(\frac{n+i-1}{m}\right),&m>n+i-1,\\
1,&m=n+i-1,
\end{cases}
\end{align*}
where $\Gamma(s)=\int^{+\infty}_{0}t^{s-1}e^{-t}\,dt$, $s>0$ is the Gamma function. Denote by $\omega_{n-1}$ the area of the surface of unit sphere in $(n-1)$-dimension. For $(z',z_{n})\in\Omega_{2R}$, denote
\begin{align}\label{ZZW666}
\delta(z'):=\varepsilon+h_{1}(z')-h(z').
\end{align}
Let $\Omega^{\ast}:=D\setminus\overline{D^{\ast}_{1}}$. We define a linear functional with respect to $\varphi$,
\begin{align}\label{linear001}
Q^{\ast}[\varphi]:=\int_{\partial D_{1}}\frac{\partial v_{0}^{\ast}}{\partial\nu},
\end{align}
where $v_{0}^{\ast}$ is a solution of the following problem:
\begin{align}\label{con003}
\begin{cases}
\Delta v_{0}^{\ast}=0,\quad\quad\;\,&\mathrm{in}\;\Omega^{\ast},\\
v_{0}^{\ast}=0,\quad\quad\;\,&\mathrm{on}\;\partial D_{1}^{\ast},\\
v_{0}^{\ast}=\varphi(x),\quad\;\,&\mathrm{on}\;\partial D.
\end{cases}
\end{align}
Note that the definition of $Q^{\ast}[\varphi]$ is valid under case ({\bf{\em S2}}) but only valid for $m<n+k-1$ under case ({\bf{\em S1}}). For $m<n-1$, define
\begin{align}\label{zz1}
a_{11}^{\ast}:=\int_{\Omega^{\ast}}|\nabla v_{1}^{\ast}|^{2},
\end{align}
where $v_{1}^{\ast}$ satisfies
\begin{equation}\label{con022}
\begin{cases}
\Delta v_{1}^{\ast}=0,\quad\;\,&\mathrm{in}\;\,\Omega^{\ast},\\
v_{1}^{\ast}=1,\quad\;\,&\mathrm{on}\;\,\partial D_{1}^{\ast}\setminus\{0\},\\
v_{1}^{\ast}=0,\quad\;\,&\mathrm{on}\;\,\partial D.
\end{cases}
\end{equation}

Unless otherwise stated, in what following $C$ represents a constant, whose values may vary from line to line, depending only on $\lambda$, $\kappa_{1},\kappa_{2},R$ and an upper bound of the $C^{2,\alpha}$ norms of $\partial D_{1}$ and $\partial D$, but not on $\varepsilon$. We also call a constant having such dependence a $universal$ $constant$. Without loss of generality, we set $\varphi(0)=0$. Otherwise, we substitute $u-\varphi(0)$ for $u$ throughout this paper. For simplicity of discussions, we assume that convexity index $m\geq2$ and growth order index $k>1$ are all positive integers in the following.
\begin{theorem}\label{thm001}
Assume that $D_{1}\subset D\subseteq\mathbb{R}^{n}\,(n\geq2)$ are defined as above, conditions ({\bf{\em H1}})--({\bf{\em H3}}) and ({\bf{\em S1}}) hold. Let $u\in H^{1}(D;\mathbb{R}^{n})\cap C^{1}(\overline{\Omega};\mathbb{R}^{n})$ be the solution of (\ref{con002}). Then for a sufficiently small $\varepsilon>0$ and $x\in\Omega_{R}$,

$(i)$ for $m\geq n+k-1$,
\begin{align*}
\nabla u=\frac{\eta\Gamma^{n+k}_{m}}{\lambda^{\frac{k}{m}}\Gamma^{n}_{m}}(1+O(r_{\varepsilon}))\rho_{k;0}(n,m;\varepsilon)\nabla\bar{u}+\nabla\bar{u}_{0}+O(\mathbf{1})\delta^{1-\frac{2}{m}}\|\varphi\|_{C^{2}(\partial D)};
\end{align*}

$(ii)$ for $n-1\leq m<n+k-1$, if $Q^{\ast}[\varphi]\neq0$,
\begin{align*}
\nabla u=\frac{m\lambda^{\frac{n-1}{m}}Q^{\ast}[\varphi]}{(n-1)\omega_{n-1}\Gamma^{n}_{m}}\frac{1+O(r_{\varepsilon})}{\rho_{0}(n,m;\varepsilon)}\nabla\bar{u}+\nabla\bar{u}_{0}+O(\mathbf{1})\delta^{1-\frac{2}{m}}\|\varphi\|_{C^{2}(\partial D)};
\end{align*}

$(iii)$ for $m<n-1$, if $Q^{\ast}[\varphi]\neq0$,
\begin{align*}
\nabla u=\frac{Q^{\ast}[\varphi]}{a_{11}^{\ast}}(1+O(r_{\varepsilon}))\nabla\bar{u}+\nabla\bar{u}_{0}+O(\mathbf{1})\delta^{1-\frac{2}{m}}\|\varphi\|_{C^{2}(\partial D)},
\end{align*}
where $\rho_{k;0}(n,m;\varepsilon)=\rho_{k}(n,m;\varepsilon)/\rho_{0}(n,m;\varepsilon)$, $\bar{u}$ and $\bar{u}_{0}$ are defined by (\ref{con009}) and (\ref{con016}), respectively, $\delta$ is defined by (\ref{ZZW666}), $a_{11}^{\ast}$ is defined by (\ref{zz1}), and
\begin{align}\label{CCC}
r_{\varepsilon}=&
\begin{cases}
\varepsilon^{\frac{1}{m}},&m>n+k,\\
\varepsilon^{\frac{1}{m}}|\ln\varepsilon|,&m=n+k,\\
|\ln\varepsilon|^{-1},&m=n+k-1,\\
\varepsilon^{\frac{n+k-1-m}{(n+k-1)(m+1)}},&n-1<m<n+k-1,\\
|\ln\varepsilon|^{-1},&m=n-1,\\
\max\{\varepsilon^{\frac{n+k-1-m}{(n+k-1)(m+1)}},\varepsilon^{\frac{1}{6}}\},&m<n-1.
\end{cases}
\end{align}

\end{theorem}

\begin{remark}
There is a great difference between interior asymptotics and boundary asymptotics. Specifically, the blow-up factor $Q_{\varepsilon}[\varphi]$ defined in \cite{LLY2019} is bounded for any boundary data $\varphi$, while $Q[\varphi]$ here can increase the singularities of the field by $\varepsilon^{\frac{n+k-1}{m}-1}$ if $m>n+k-1$ or $|\ln\varepsilon|$ if $m=n+k-1$ for the boundary data $\varphi$ with $k$-order growth. In addition, when $m>2$, the remainder of order $O(\varepsilon^{1-2/m})$ in the shortest line segment between the conductors and the matrix boundary provides a more precise characterization on the asymptotic behavior of the concentration than that of $m=2$. Finally, the concisely main terms $\nabla\bar{u}$ and $\nabla\bar{u}_{0}$ together with their coefficients can completely describe the singular effect of the geometry, which will greatly reduce the complexity of numerical computation for $\nabla u$.
\end{remark}

\begin{remark}
The asymptotics of $\nabla u$ in Theorem \ref{thm001} indicate that

$(1)$ if $m\leq n+k-1$, then its maximum achieves only at $\{x'=0'\}\cap\Omega$;

$(2)$ if $m>n+k-1$, then the maximum achieves at both $\{x'=0'\}\cap\Omega$ and $\{|x'|=\varepsilon^{\frac{1}{m}}\}\cap\Omega$.
\end{remark}

\begin{remark}
In order to further reveal the effect of principal curvatures of the geometry, we take $n=3$ relevant to physical dimension for example. Consider
\begin{align*}
\varphi=\eta_{1}|x_{1}|^{k}+\eta_{2}|x_{2}|^{k},\quad x\in\{\lambda_{1}|x_{1}|^{m}+\lambda_{2}|x_{2}|^{m}<R,\;x_{3}=h(x')\},
\end{align*}
and
\begin{align}\label{Geometry}
h_{1}(x')-h(x')=\lambda_{1}|x_{1}|^{m}+\lambda_{2}|x_{2}|^{m},\quad x'\in\{\lambda_{1}|x_{1}|^{m}+\lambda_{2}|x_{2}|^{m}<R\},
\end{align}
where $\lambda_{i},\eta_{i}$, $i=1,2$, are four positive constant independent of $\varepsilon$. Then by the same method as in Theorem \ref{thm001}, we find that the coefficient of the main term $\nabla\bar{u}$ has an explicit dependence on $\lambda_{i}$ and $\eta_{i}$ in the form of that $\eta_{1}\lambda_{1}^{-\frac{k}{m}}+\eta_{2}\lambda_{2}^{-\frac{k}{m}}$ for $m\geq k+2$ and $\sqrt[m]{\lambda_{1}\lambda_{2}}$ for $m<k+2$.

\end{remark}

\begin{theorem}\label{coro002}
Assume that $D_{1}\subset D\subseteq\mathbb{R}^{n}\,(n\geq2)$ are defined as above, conditions ({\bf{\em H1}})--({\bf{\em H3}}) and ({\bf{\em S2}}) hold, $Q^{\ast}[\varphi]\neq0$. Let $u\in H^{1}(D;\mathbb{R}^{n})\cap C^{1}(\overline{\Omega};\mathbb{R}^{n})$ be the solution of (\ref{con002}).  Then for a sufficiently small $\varepsilon>0$,

$(i)$ for $m\geq n-1$,
\begin{align*}
\nabla u=\frac{m\lambda^{\frac{n-1}{m}}Q^{\ast}[\varphi]}{(n-1)\omega_{n-1}\Gamma^{n}_{m}}\frac{1+O(\tilde{r}_{\varepsilon})}{\rho_{0}(n,m;\varepsilon)}\nabla\bar{u}+\nabla\bar{u}_{0}+O(\mathbf{1})\delta^{1-\frac{2}{m}}\|\varphi\|_{C^{2}(\partial D)};
\end{align*}

$(ii)$ for $m<n-1$,

\begin{align*}
\nabla u=\frac{Q^{\ast}[\varphi]}{a_{11}^{\ast}}(1+O(\tilde{r}_{\varepsilon}))\nabla\bar{u}+\nabla\bar{u}_{0}+O(\mathbf{1})\delta^{1-\frac{2}{m}}\|\varphi\|_{C^{2}(\partial D)},
\end{align*}
where $\bar{u}$ and $\bar{u}_{0}$ are defined by (\ref{con009}) and (\ref{con016}), respectively, $\delta$ is defined by (\ref{ZZW666}), $a_{11}^{\ast}$ is defined by (\ref{zz1}), and
\begin{align*}
\tilde{r}_{\varepsilon}=&
\begin{cases}
\varepsilon^{\frac{m+n-2}{(m+1)(2m+n-2)}},&m>n-1,\\
|\ln\varepsilon|^{-1},&m=n-1,\\
\max\{\varepsilon^{\frac{m+n-2}{(m+1)(2m+n-2)}},\varepsilon^{\frac{1}{6}}\}.&m<n-1.
\end{cases}
\end{align*}
\end{theorem}

\begin{remark}
The asymptotics of $\nabla u$ in Theorem \ref{coro002} imply that

$(1)$ if $m<n$, then its maximum attains only at $\{x'=0'\}\cap\Omega$;

$(2)$ if $m=n$, then the maximum attains at $\{x'=0'\}\cap\Omega$ and $\{|x'|=\varepsilon^{\frac{1}{m}}\}\cap\Omega$ simultaneously;

$(3)$ if $m>n$, then the maximum attains at $\{|x'|=\varepsilon^{\frac{1}{m}}\}\cap\Omega$.
\end{remark}

\begin{remark}
If (\ref{Geometry}) holds in Theorem \ref{coro002}, we can obtain that the coefficient of the main term $\nabla\bar{u}$ has an explicit dependence of $\sqrt[m]{\lambda_{1}\lambda_{2}}$.
\end{remark}

The organization of this paper is as follows. In section 2, we carry out a linear decomposition of the solution $u$ to problem (\ref{con002}) as $v_{0}$ and $v_{1}$, defined by (\ref{con005}) and (\ref{con006}) below, and we prove the correspondingly main terms $\bar{u}_{0}$ and $\bar{u}$ constructed by (\ref{con009}) and (\ref{con016}), respectively, in Lemma \ref{lem001} and Theorem \ref{thm002}. Based on the results obtained in section 2, we give the proofs of Theorem \ref{thm001} and Theorem \ref{coro002} consisting of the asymptotics of blow-up factor $Q[\varphi]$ and $a_{11}$ in section 3.

\section{Preliminary}

\subsection{Solution split}
As in \cite{LX2017}, we decompose the solution $u$ of (\ref{con002}) as follows
\begin{align}\label{con0033}
u(x)=C_{1}v_{1}(x)+v_{0}(x),\quad\;\,\mathrm{in}\;D\setminus\overline{D}_{1},
\end{align}
where $v_{i}$, $i=0,1$, verify
\begin{align}\label{con005}
\begin{cases}
\Delta v_{0}=0,\quad\quad\;\,&\mathrm{in}\;\Omega,\\
v_{0}=0,\quad\quad\;\,&\mathrm{on}\;\partial D_{1},\\
v_{0}=\varphi(x),\quad\;\,&\mathrm{on}\;\partial D,
\end{cases}
\end{align}
and
\begin{align}\label{con006}
\begin{cases}
\Delta v_{1}=0,\quad\quad\;\,&\mathrm{in}\;\Omega,\\
v_{1}=1,\quad\quad\;\,&\mathrm{on}\;\partial D_{1},\\
v_{1}=0,\quad\;\,&\mathrm{on}\;\partial D,
\end{cases}
\end{align}
respectively. Similarly as (\ref{linear001}) and (\ref{con003}), we define a linear functional of $\varphi$ as follows
\begin{align}\label{linear002}
Q[\varphi]=\int_{\partial D_{1}}\frac{\partial v_{0}}{\partial\nu},
\end{align}
where $v_{0}$ is defined by (\ref{con005}). Denote
\begin{align*}
a_{11}:=\int_{\Omega}|\nabla v_{1}|^{2}dx.
\end{align*}
Then, it follows from the third line of (\ref{con002}) and the decomposition (\ref{con0033}) that
\begin{align*}
C_{1}\int_{\partial D_{1}}\frac{\partial v_{1}}{\partial\nu}+\int_{\partial D_{1}}\frac{\partial v_{0}}{\partial\nu}=0.
\end{align*}
Recalling the definition of $v_{1}$ and making use of integration by parts, we have
\begin{align}\label{con007}
\nabla u=\frac{Q[\varphi]}{a_{11}}\nabla v_{1}+\nabla v_{0}.
\end{align}

\subsection{A general boundary value problem}

To obtain the asymptotic expansion for $\nabla u$, we first consider the following general boundary value problem:
\begin{equation}\label{con008}
\begin{cases}
\Delta v=0,\quad\;\,&\mathrm{in}\;\,\Omega,\\
v=\psi,&\mathrm{on}\;\,\partial D_{1},\\
v=0,&\mathrm{on}\;\,\partial D,
\end{cases}
\end{equation}
where $\psi\in C^{2}(\partial D_{1})$ is a given scalar function. Note that if $\psi=1$ on $\partial D_{1}$, then $v_{1}=v$. Extend $\psi\in C^{2}(\partial D_{1})$ to $\psi\in C^{2}(\overline{\Omega})$ such that $\|\psi\|_{C^{2}(\overline{\Omega\setminus\Omega_{R}})}\leq C\|\psi\|_{C^{2}(\partial D_{1})}$. Construct a cutoff function $\rho\in C^{2}(\overline{\Omega})$ satisfying $0\leq\rho\leq1$, $|\nabla\rho|\leq C$ on $\overline{\Omega}$, and
\begin{align}\label{con011}
\rho=1\;\,\mathrm{on}\;\,\Omega_{\frac{3}{2}R},\quad\rho=0\;\,\mathrm{on}\;\,\overline{\Omega}\setminus\Omega_{2R}.
\end{align}
For $x\in\Omega$, we define
\begin{align*}
\bar{v}(x)=[\rho(x)\psi(x',\varepsilon+h_{1}(x'))+(1-\rho(x))\psi(x)]\bar{u}(x),
\end{align*}
where $\bar{u}$ is defined by (\ref{con009}). Specially,
\begin{align*}
\bar{v}(x)=\psi(x',\varepsilon+h_{1}(x'))\bar{u}(x),\quad\;\,\mathrm{in}\;\Omega_{R}.
\end{align*}
Due to (\ref{con009}), we have
\begin{align}\label{KK6}
\|\bar{v}\|_{C^{2}(\Omega\setminus\Omega_{R})}\leq C\|\psi\|_{C^{2}(\partial D_{1})}.
\end{align}

Similarly as in \cite{LX2017}, we can obtain an asymptotic expansion of the gradient for problem (\ref{con006}).
\begin{theorem}\label{thm002}
Assume as above. Let $v\in H^{1}(\Omega)$ be a weak solution of (\ref{con008}). Then, for a sufficiently small $\varepsilon>0$,
\begin{align}\label{con013}
|\nabla(v-\bar{v})(x)|\leq C\delta^{1-\frac{2}{m}}(|\psi(x',\varepsilon+h_{1}(x'))|+\delta^{\frac{1}{m}}\|\psi\|_{C^{2}(\partial D_{1})}),\quad\mathrm{in}\;\,\Omega_{R}.
\end{align}
Consequently, (\ref{con013}), together with choosing $\psi=1$ on $\partial D_{1}$, yields that
\begin{align}\label{con015}
\nabla v_{1}=\nabla\bar{u}+O(\mathbf{1})\delta^{1-\frac{2}{m}},\quad\;\,\mathrm{in}\;\Omega_{R},
\end{align}
and
\begin{align*}
\|\nabla v\|_{L^{\infty}(\Omega\setminus\Omega_{R})}\leq C\|\psi\|_{C^{2}(\partial D_{1})}.
\end{align*}
where $v_{1}\in H^{1}(\Omega)$ is a weak solution of (\ref{con006})

\end{theorem}
Note that when $m>2$, the remainder of order $O(1)$ in \cite{LX2017} is improved to that of order $O(\varepsilon^{1-2/m})$ for $x\in\{x'=0'\}\cap\Omega_{R}$ here. For readers' convenience, the detailed proof of Theorem \ref{thm002} is left in the Appendix. Similarly, by applying Theorem \ref{thm002}, we can find that the leading term of $\nabla v_{0}$ is $\nabla\bar{u}_{0}$ in the following.
\begin{lemma}\label{lem001}
Assume as above. Let $v_{0}$ be the weak solution of (\ref{con005}). Then, for a sufficiently small $\varepsilon>0$,
\begin{align}\label{con018}
\nabla v_{0}=\nabla\bar{u}_{0}+O(\mathbf{1})\delta^{1-\frac{2}{m}}(|\varphi(x',h(x'))|+\delta^{\frac{1}{m}}\|\varphi\|_{C^{2}(\partial D)}),\quad\;\,\mathrm{in}\;\Omega_{R},
\end{align}
and
\begin{align}\label{con01818}
\|\nabla_{x'}v_{0}\|_{L^{\infty}(\Omega_{R})}\leq C\|\varphi\|_{C^{2}(\partial D)},\;\,\|\nabla v_{0}\|_{L^{\infty}(\Omega\setminus\Omega_{R})}\leq C\|\varphi\|_{C^{2}(\partial D)},
\end{align}
where $\bar{u}_{0}$ is defined by (\ref{con016}).
\end{lemma}

Therefore, recalling the decomposition (\ref{con007}) and in view of (\ref{con015}) and (\ref{con018}), for the purpose of deriving the asymptotic of $\nabla u$, it suffices to establish the following two aspects of expansions:

(i) Expansion of $Q[\varphi]$;

(ii) Expansion of $a_{11}$.

\section{Proofs of Theorem \ref{thm001} and Theorem \ref{coro002}}

\subsection{Expansion of $Q[\varphi]$}
Before proving Theorem \ref{thm001} and Theorem \ref{coro002}, we first give an expansion of $Q[\varphi]$ with respect to $\varepsilon$.
\begin{lemma}\label{lem002}
Assume as above. Then, for a sufficiently small $\varepsilon>0$,

$(a)$ if ({\bf{\em S1}}) holds for $m\geq n+k-1$ in Theorem \ref{thm001},
\begin{align*}
Q[\varphi]=&\frac{(n-1)\omega_{n-1}\eta\Gamma^{n+k}_{m}}{m\lambda^{\frac{n+k-1}{m}}}\rho_{k}(n,m;\varepsilon)
\begin{cases}
1+O(1)\varepsilon^{\frac{1}{m}},&m>n+k,\\
1+O(1)\varepsilon^{\frac{1}{m}}|\ln\varepsilon|,&m=n+k,\\
1+O(1)|\ln\varepsilon|^{-1},&m=n+k-1;
\end{cases}
\end{align*}

$(b)$ if ({\bf{\em S1}}) holds for $m<n+k-1$ in Theorem \ref{thm001},
\begin{align*}
Q[\varphi]=&
Q^{\ast}[\varphi]+O(1)\varepsilon^{\frac{n+k-1-m}{(n+k-1)(m+1)}}.
\end{align*}

$(c)$ if ({\bf{\em S2}}) holds in Theorem \ref{coro002},
\begin{align*}
Q[\varphi]=&
Q^{\ast}[\varphi]+O(1)\varepsilon^{\frac{m+n-2}{(m+1)(2m+n-2)}}.
\end{align*}

\end{lemma}

\begin{proof}
{\bf Step 1.} Proof of $(a)$. Note that the unit outward normal $\nu$ to $\partial D_{1}$ is given by
$$\nu=\frac{(\nabla_{x'}h_{1}(x'),-1)}{\sqrt{1+|\nabla_{x'}h_{1}(x')|^{2}}},\quad\;\,\mathrm{in}\;\Omega_{R}.$$
In light of ({\bf{\em H2}}), we obtain that for $i=1,\cdots,n-1$,
\begin{align}\label{con021}
|\nu_{i}|\leq C|x'|^{m-1},\quad|\nu_{n}|\leq1,\quad\;\,\mathrm{in}\;\Omega_{R}.
\end{align}
Recalling the definition of $Q[\varphi]$, it follows from (\ref{con018})--(\ref{con01818}) and (\ref{con021}) that
\begin{align*}
Q[\varphi]=&\int_{\partial D_{1}}\partial_{n}v_{0}\nu_{n}+\int_{\partial D_{1}}\sum^{n-1}_{i=1}\partial_{i}v_{0}\nu_{i}\\
=&\int_{|x'|<R}\frac{\eta|x'|^{k}}{\varepsilon+\lambda|x'|^{m}}+O(1)\int_{|x'|<R}\frac{\eta|x'|^{k+1}}{\varepsilon+\lambda|x'|^{m}}+O(1)\|\varphi\|_{C^{2}(\partial D)}\\
=&\frac{(n-1)\omega_{n-1}\eta\Gamma^{n+k}_{m}}{m\lambda^{\frac{n+k-1}{m}}}\begin{cases}
\varepsilon^{\frac{n+k-1}{m}-1}+O(1)\varepsilon^{\frac{n+k}{m}-1}\|\varphi\|_{C^{2}(\partial D)},&m>n+k,\\
\varepsilon^{-\frac{1}{m}}+O(1)|\ln\varepsilon|\|\varphi\|_{C^{2}(\partial D)},&m=n+k,\\
|\ln\varepsilon|+O(1)\|\varphi\|_{C^{2}(\partial D)},&m=n+k-1.
\end{cases}
\end{align*}

{\bf Step 2.} Proofs of $(b)$ and $(c)$. In view of the definitions of $Q[\varphi]$ and $Q^{\ast}[\varphi]$, it follows from integration by parts that
\begin{align*}
Q[\varphi]=\int_{\partial D}\frac{\partial v_{1}}{\partial\nu}\varphi(x),\quad\quad Q^{\ast}[\varphi]=\int_{\partial D}\frac{\partial v_{1}^{\ast}}{\partial\nu}\varphi(x),
\end{align*}
where $v_{1}$ and $v_{1}^{\ast}$ are defined by (\ref{con006}) and (\ref{con022}). Thus,
\begin{align*}
Q[\varphi]-Q^{\ast}[\varphi]=\int_{\partial D}\frac{\partial(v_{1}-v_{1}^{\ast})}{\partial\nu}\cdot\varphi(x).
\end{align*}

To estimate $v_{1}-v_{1}^{\ast}$, we first introduce a scar auxiliary functions $\bar{u}^{\ast}$ satisfying $\bar{u}^{\ast}=1$ on $\partial D_{1}^{\ast}\setminus\{0\}$, $\bar{u}^{\ast}=0$ on $\partial D$, and
$$\bar{u}^{\ast}=\frac{x_{n}-h(x')}{h_{1}(x')-h(x')},\quad\mathrm{in}\;\,\Omega_{2R}^{\ast},\quad\;\,\|\bar{u}^{\ast}\|_{C^{2}(\Omega^{\ast}\setminus\Omega_{R}^{\ast})}\leq C,$$
where $\Omega^{\ast}_{r}:=\Omega^{\ast}\cap\{|x'|<r\},$ $0<r\leq2R$. In view of ({\bf{\em H2}}), we obtain that for $x\in\Omega_{R}^{\ast}$,
\begin{align}\label{con023}
|\nabla_{x'}(\bar{u}-\bar{u}^{\ast})|\leq\frac{C}{|x'|},
\end{align}
and
\begin{align}\label{con025}
|\partial_{n}(\bar{u}-\bar{u}^{\ast})|\leq\frac{C\varepsilon}{|x'|^{m}(\varepsilon+|x'|^{m})}.
\end{align}
Applying Theorem \ref{thm002} to (\ref{con022}), it follows that for $x\in\Omega_{R}^{\ast}$,
\begin{align}\label{con026}
|\nabla(v_{1}^{\ast}-\bar{u}^{\ast})|\leq C|x'|^{m-2},
\end{align}
and
\begin{align}\label{con027}
|\nabla_{x'}v_{1}^{\ast}|\leq\frac{C}{|x'|},\quad\;|\partial_{n}v_{1}^{\ast}|\leq\frac{C}{|x'|^{m}}.
\end{align}
For $0<r<R$, denote
\begin{align}\label{con028}
\mathcal{C}_{r}:=\left\{x\in\mathbb{R}^{n}\Big|\;|x'|<r,\,\frac{1}{2}\min_{|x'|\leq r}h(x')\leq x_{n}\leq\varepsilon+2\max_{|x'|\leq r}h_{1}(x')\right\}.
\end{align}
We now divide into two steps to estimate $|Q[\varphi]-Q^{\ast}[\varphi]|$.

{\bf Step 2.1.} Note that $v_{1}-v_{1}^{\ast}$ solves
\begin{align*}
\begin{cases}
\Delta(v_{1}-v_{1}^{\ast})=0,&\mathrm{in}\;\,D\setminus(\overline{D_{1}\cup D_{1}^{\ast}}),\\
v_{1}-v_{1}^{\ast}=1-v_{1}^{\ast},&\mathrm{on}\;\,\partial D_{1}\setminus D_{1}^{\ast},\\
v_{1}-v_{1}^{\ast}=v_{1}-1,&\mathrm{on}\;\,\partial D_{1}^{\ast}\setminus(D_{1}\cup\{0\}),\\
v_{1}-v_{1}^{\ast}=0,&\mathrm{on}\;\,\partial D.
\end{cases}
\end{align*}
We first estimate $|v_{1}-v_{1}^{\ast}|$ on $\partial(D_{1}\cup D_{1}^{\ast})\setminus\mathcal{C}_{\varepsilon^{\gamma}}$, where $0<\gamma<1/2$ to be determined later. In light of the definition of $v_{1}^{\ast}$, we derive that
$$|\partial_{n}v_{1}^{\ast}|\leq C,\quad\;\,\mathrm{in}\;\Omega^{\ast}\setminus\Omega^{\ast}_{R}.$$
Therefore,
\begin{align}\label{con029}
|v_{1}-v_{1}^{\ast}|\leq C\varepsilon,\quad\;\,\mathrm{for}\;\,x\in\partial D_{1}\setminus D_{1}^{\ast}.
\end{align}
It follows from (\ref{con015}) that
\begin{align}\label{con030}
|v_{1}-v_{1}^{\ast}|\leq C\varepsilon^{1-m\gamma},\quad\;\,\mathrm{on}\;\,\partial D_{1}^{\ast}\setminus(D_{1}\cup\mathcal{C}_{\varepsilon^{\gamma}}).
\end{align}
Combining Theorem \ref{thm002} and (\ref{con025})--(\ref{con026}), we obtain that for $x\in\Omega_{R}^{\ast}\cap\{|x'|=\varepsilon^{\gamma}\}$,
\begin{align*}
|\partial_{n}(v_{1}-v_{1}^{\ast})|\leq&|\partial_{n}(v_{1}-\bar{u})|+|\partial_{n}(\bar{u}-\bar{u}^{\ast})|+|\partial_{n}(v_{1}^{\ast}-\bar{u}^{\ast})|\\
\leq&C\left(\frac{1}{\varepsilon^{2m\gamma-1}}+\varepsilon^{(m-2)\gamma}\right),
\end{align*}
which together with $v_{1}-v_{1}^{\ast}=0$ on $\partial D$ yields that
\begin{align}\label{con031}
|(v_{1}-v_{1}^{\ast})(x',x_{n})|=&|(v_{1}-v_{1}^{\ast})(x',x_{n})-(v_{1}-v_{1}^{\ast})(x',h(x'))|\notag\\
\leq&C\big(\varepsilon^{1-m\gamma}+\varepsilon^{2(m-1)\gamma}\big).
\end{align}
Take $\gamma=\frac{1}{m+1}$. Then, it follows from (\ref{con029})--(\ref{con031}) that
$$|v_{1}-v_{1}^{\ast}|\leq C\varepsilon^{\frac{1}{m+1}},\quad\;\,\mathrm{on}\;\,\partial\big(D\setminus\big(\overline{D_{1}\cup D_{1}^{\ast}\cup\mathcal{C}_{\varepsilon^{\frac{1}{m+1}}}}\big)\big).$$
Making use of the maximum principle, we obtain
$$|v_{1}-v_{1}^{\ast}|\leq C\varepsilon^{\frac{1}{m+1}},\quad\;\,\mathrm{in}\;\,D\setminus\big(\overline{D_{1}\cup D_{1}^{\ast}\cup\mathcal{C}_{\varepsilon^{\frac{1}{m+1}}}}\big).$$
This, together with the standard interior and boundary estimates, leads to that, for any $\frac{m-1}{m(m+1)}<\tilde{\gamma}<\frac{1}{m+1}$,
$$|\nabla(v_{1}-v_{1}^{\ast})|\leq C\varepsilon^{m\tilde{\gamma}-\frac{m-1}{m+1}},\quad\;\,\mathrm{in}\;\,D\setminus\big(\overline{D_{1}\cup D_{1}^{\ast}\cup\mathcal{C}_{\varepsilon^{\frac{1}{m+1}-\tilde{\gamma}}}}\big),$$
which implies that
\begin{align}\label{con032}
|\mathcal{A}^{out}|:=\left|\int_{\partial D\setminus\mathcal{C}_{\varepsilon^{\frac{1}{m+1}-\tilde{\gamma}}}}\frac{\partial(v_{1}-v_{1}^{\ast})}{\partial\nu}\cdot\varphi(x)\right|\leq C\|\varphi\|_{L^{\infty}(\partial D)}\varepsilon^{m\tilde{\gamma}-\frac{m-1}{m+1}},
\end{align}
where $\frac{m-1}{m(m+1)}<\tilde{\gamma}<\frac{1}{m+1}$ to be determined later.

{\bf Step 2.2.} We further estimate
\begin{align*}
\mathcal{A}^{in}:=&\int_{\partial D\cap\mathcal{C}_{\varepsilon^{\frac{1}{m+1}-\tilde{\gamma}}}}\frac{\partial(v_{1}-v_{1}^{\ast})}{\partial\nu}\cdot\varphi(x)\\
=&\int_{\partial D\cap\mathcal{C}_{\varepsilon^{\frac{1}{m+1}-\tilde{\gamma}}}}\frac{\partial(w_{1}-w_{1}^{\ast})}{\partial\nu}\cdot\varphi(x)+\int_{\partial D\cap\mathcal{C}_{\varepsilon^{\frac{1}{m+1}-\tilde{\gamma}}}}\frac{\partial(\bar{u}-\bar{u}^{\ast})}{\partial\nu}\cdot\varphi(x)\\
=&:\mathcal{A}_{w}+\mathcal{A}_{u},
\end{align*}
where $w_{1}=v_{1}-\bar{u}$ and $w_{1}^{\ast}=v_{1}^{\ast}-\bar{u}^{\ast}$. To begin with, applying Theorem \ref{thm002}, we obtain that
\begin{align}\label{ZW12}
|\mathcal{A}_{w}|\leq C\eta\int_{\partial D\cap\mathcal{C}_{\varepsilon^{\frac{1}{m+1}-\tilde{\gamma}}}}|x'|^{m+k-2}\leq C\eta\varepsilon^{(\frac{1}{m+1}-\tilde{\gamma})(m+n+k-3)}.
\end{align}

To estimate $\mathcal{A}_{u}$, we split it into two parts as follows.
\begin{align*}
\mathcal{A}_{u}=&\int_{\partial D\cap\mathcal{C}_{\varepsilon^{\frac{1}{m+1}-\tilde{\gamma}}}}\sum^{n-1}_{i=1}\partial_{i}(\bar{u}-\bar{u}^{\ast})\nu_{i}\varphi(x)+\int_{\partial D\cap\mathcal{C}_{\varepsilon^{\frac{1}{m+1}-\tilde{\gamma}}}}\partial_{n}(\bar{u}-\bar{u}^{\ast})\nu_{n}\varphi(x)\\
=&:\mathcal{A}^{1}_{u}+\mathcal{A}^{2}_{u}.
\end{align*}

{\bf Case 1.} If ({\bf{\em S1}}) holds for $m<n+k-1$ in Theorem \ref{thm001}, owing to (\ref{con021}) and (\ref{con023})--(\ref{con025}), we obtain that
\begin{align*}
|\mathcal{A}^{1}_{u}|\leq C\eta\varepsilon^{\left(\frac{1}{m+1}-\tilde{\gamma}\right)(n+k+m-3)},\;\,|\mathcal{A}^{2}_{u}|\leq C\eta\varepsilon^{\left(\frac{1}{m+1}-\tilde{\gamma}\right)(n+k-m-1)}.
\end{align*}
Then
\begin{align*}
|\mathcal{A}_{u}|\leq C\eta\varepsilon^{\left(\frac{1}{m+1}-\tilde{\gamma}\right)(n+k-m-1)}.
\end{align*}
This, together with (\ref{con032})--(\ref{ZW12}) and picking $\tilde{\gamma}=\frac{n+k-2}{(n+k-1)(m+1)}$, yields that
\begin{align*}
|Q[\varphi]-Q^{\ast}[\varphi]|\leq&C(\eta+\|\varphi\|_{L^{\infty}(\partial D)})\varepsilon^{\frac{n+k-m-1}{(n+k-1)(m+1)}}.
\end{align*}

{\bf Case 2.} If ({\bf{\em S2}}) holds in Theorem \ref{coro002}, based on the fact that the integrating domain is symmetric with respect to $x_{i}$, $i=1,\cdots,n-1$, we have
\begin{align*}
|\mathcal{A}^{1}_{u}|\leq C\eta\varepsilon^{\left(\frac{1}{m+1}-\tilde{\gamma}\right)(n+m-2)},\;\,\mathcal{A}^{2}_{u}=0.
\end{align*}
Hence,
\begin{align*}
|\mathcal{A}_{u}|\leq C\eta\varepsilon^{\left(\frac{1}{m+1}-\tilde{\gamma}\right)(n+m-2)}.
\end{align*}
This, together with (\ref{con032})--(\ref{ZW12}) and taking $\tilde{\gamma}=\frac{2m+n-3}{(2m+n-2)(m+1)}$, leads to that
\begin{align*}
|Q[\varphi]-Q^{\ast}[\varphi]|\leq C(\eta+\|\varphi\|_{L^{\infty}(\partial D)})\varepsilon^{\frac{m+n-2}{(m+1)(2m+n-2)}}.
\end{align*}

Consequently, it follows from {\bf Step 1} and {\bf Step 2} that Lemma \ref{lem002} holds.

\end{proof}

\subsection{Expansion of $a_{11}$}
Before stating the asymptotic of $a_{11}$ with respect to $\varepsilon$, we first introduce a notation used in the following. Denote
\begin{align}
A:=&\int_{\Omega^{\ast}\setminus\Omega_{R}^{\ast}}|\nabla v_{1}^{\ast}|^{2}+2\int_{\Omega_{R}^{\ast}}\nabla\bar{u}^{\ast}\cdot\nabla(v_{1}^{\ast}-\bar{u}^{\ast})\notag\\
&+\int_{\Omega^{\ast}_{R}}\big(|\nabla(v_{1}^{\ast}-\bar{u}^{\ast})|^{2}+|\partial_{x'}\bar{u}^{\ast}|^{2}\big).\label{con03333}
\end{align}

\begin{lemma}\label{lem003}
Assume as in Theorem \ref{thm001} and Theorem \ref{coro002}. Then, for a sufficiently small $\varepsilon>0$,

$(i)$ for $m\geq n-1$,
\begin{align*}
a_{11}=&
\frac{(n-1)\omega_{n-1}\Gamma^{n}_{m}}{m\lambda^{\frac{n-1}{m}}}\rho_{0}(n,m;\varepsilon)
\begin{cases}
1+O(1)\varepsilon^{\frac{1}{m}},&m>n,\\
1+O(1)\varepsilon^{\frac{1}{m}}|\ln\varepsilon|,&m=n,\\
1+O(1)|\ln\varepsilon|^{-1},&m=n-1;
\end{cases}
\end{align*}

$(ii)$ for $m<n-1$,
\begin{align*}
a_{11}=&a_{11}^{\ast}+O(1)\varepsilon^{\frac{1}{6}},
\end{align*}
where $a_{11}^{\ast}$ is defined by (\ref{zz1}).

\end{lemma}

\begin{proof}
Fix $\bar{\gamma}=\frac{1}{6m}$. We first split $a_{11}$ into three parts as follows.
\begin{align*}
a_{11}=\int_{\Omega_{\varepsilon^{\bar{\gamma}}}}|\nabla v_{1}|^{2}+\int_{\Omega_{R}\setminus\Omega_{\varepsilon^{\bar{\gamma}}}}|\nabla v_{1}|^{2}+\int_{\Omega\setminus\Omega_{R}}|\nabla v_{1}|^{2}=:\mathrm{I}+\mathrm{II}+\mathrm{III}.
\end{align*}

{\bf Step 1.} As for $\mathrm{I}$, recalling the definition of $\bar{u}$ and using Theorem \ref{thm002}, we obtain that
\begin{align}\label{con03365}
\mathrm{I}=&\int_{\Omega_{\varepsilon^{\bar{\gamma}}}}|\partial_{n}\bar{u}|^{2}+\int_{\Omega_{\varepsilon^{\bar{\gamma}}}}|\partial_{x'}\bar{u}|^{2}+2\int_{\Omega_{\varepsilon^{\bar{\gamma}}}}\nabla\bar{u}\cdot\nabla(v_{1}-\bar{u})+\int_{\Omega_{\varepsilon^{\bar{\gamma}}}}|\nabla(v_{1}-\bar{u})|^{2}\notag\\
=&\int_{|x'|<\varepsilon^{\bar{\gamma}}}\frac{dx'}{\varepsilon+h_{1}(x')-h(x')}+O(1)\varepsilon^{\frac{n+m-3}{6m}}.
\end{align}

For the second term $\mathrm{II}$, we further decompose it into three parts as follows
\begin{align*}
\mathrm{II}_{1}=&\int_{(\Omega_{R}\setminus\Omega_{\varepsilon^{\bar{\gamma}}})\setminus(\Omega^{\ast}_{R}\setminus\Omega^{\ast}_{\varepsilon^{\bar{\gamma}}})}|\nabla v_{1}|^{2},\\
\mathrm{II}_{2}=&\int_{\Omega^{\ast}_{R}\setminus\Omega^{\ast}_{\varepsilon^{\bar{\gamma}}}}|\nabla(v_{1}-v_{1}^{\ast})|^{2}+2\int_{\Omega^{\ast}_{R}\setminus\Omega^{\ast}_{\varepsilon^{\bar{\gamma}}}}\nabla v_{1}^{\ast}\cdot\nabla(v_{1}-v_{1}^{\ast}),\\
\mathrm{II}_{3}=&\int_{\Omega^{\ast}_{R}\setminus\Omega^{\ast}_{\varepsilon^{\bar{\gamma}}}}|\nabla v_{1}^{\ast}|^{2}.
\end{align*}
Due to the fact that the thickness of $(\Omega_{R}\setminus\Omega_{\varepsilon^{\bar{\gamma}}})\setminus(\Omega^{\ast}_{R}\setminus\Omega^{\ast}_{\varepsilon^{\bar{\gamma}}})$ is $\varepsilon$, it follows from (\ref{con015}) that
\begin{align}\label{con0333355}
\mathrm{II}_{1}\leq&C\varepsilon\int_{\varepsilon^{\bar{\gamma}}<|x'|<R}\frac{dx'}{|x'|^{2m}}\leq C
\begin{cases}
\varepsilon^{\frac{4m+n-1}{6m}},&m>\frac{n-1}{2},\\
\varepsilon|\ln\varepsilon|,&m=\frac{n-1}{2},\\
\varepsilon,&m<\frac{n-1}{2}.
\end{cases}
\end{align}
By picking $\gamma=\frac{1}{2m}$ in {\bf Step 2.1} of the proof of Lemma \ref{lem002}, it follows from (\ref{con029})--(\ref{con031}) and the maximum principle that
\begin{align*}
|v_{1}-v_{1}^{\ast}|\leq C\varepsilon^{\frac{1}{2}},\quad\;\,\mathrm{in}\;\,D\setminus\big(\overline{D_{1}\cup D_{1}^{\ast}\cup\mathcal{C}_{\varepsilon^{\frac{1}{2m}}}}\big).
\end{align*}
Similarly as before, utilizing the standard interior and boundary estimates, we derive that
\begin{align}\label{con035}
|\nabla(v_{1}-v_{1}^{\ast})|\leq C\varepsilon^{\frac{1}{6}},\quad\;\,\mathrm{in}\;\,D\setminus\big(\overline{D_{1}\cup D_{1}^{\ast}\cup\mathcal{C}_{\varepsilon^{\frac{1}{3m}}}}\big).
\end{align}
Then combining (\ref{con027}) and (\ref{con035}), we obtain that
\begin{align}\label{con036}
|\mathrm{II}_{2}|\leq&C\varepsilon^{\frac{1}{6}}.
\end{align}

As for $\mathrm{II}_{3}$, it follows from (\ref{con026}) and (\ref{con027}) that
\begin{align*}
\mathrm{II}_{3}=&\int_{\Omega^{\ast}_{R}\setminus\Omega^{\ast}_{\varepsilon^{\bar{\gamma}}}}|\nabla\bar{u}^{\ast}|^{2}+2\int_{\Omega^{\ast}_{R}\setminus\Omega^{\ast}_{\varepsilon^{\bar{\gamma}}}}\nabla\bar{u}^{\ast}\cdot\nabla(v_{1}^{\ast}-\bar{u}^{\ast})+\int_{\Omega^{\ast}_{R}\setminus\Omega^{\ast}_{\varepsilon^{\bar{\gamma}}}}|\nabla(v_{1}^{\ast}-\bar{u}^{\ast})|^{2}\\
=&\int_{\varepsilon^{\bar{\gamma}}<|x'|<R}\frac{dx'}{h_{1}(x')-h(x')}+A-\int_{\Omega^{\ast}\setminus\Omega_{R}^{\ast}}|\nabla v_{1}^{\ast}|^{2}+O(1)\varepsilon^{(n+m-3)\bar{\gamma}},
\end{align*}
where $A$ is defined by (\ref{con03333}). This, together with (\ref{con0333355}) and (\ref{con036}), leads to that
\begin{align}\label{con037}
\mathrm{II}=&\int_{\varepsilon^{\bar{\gamma}}<|x'|<R}\frac{dx'}{h_{1}(x')-h(x')}+A-\int_{\Omega^{\ast}\setminus\Omega_{R}^{\ast}}|\nabla v_{1}^{\ast}|^{2}\notag\\
&+O(1)
\begin{cases}
\varepsilon^{\frac{m-1}{6m}},&n=2,\\
\varepsilon^{\frac{1}{6}},&n\geq3.
\end{cases}
\end{align}

For the last term $\mathrm{III}$, due to the fact that $|\nabla v_{1}|$ is bounded in $D_{1}^{\ast}\setminus(D_{1}\cup\Omega_{R})$ and $D_{1}\setminus D_{1}^{\ast}$ and the fact that the volume of $D_{1}^{\ast}\setminus(D_{1}\cup\Omega_{R})$ and $D_{1}\setminus D_{1}^{\ast}$ is of order $O(\varepsilon)$, it follows from (\ref{con035}) that
\begin{align*}
\mathrm{III}=&\int_{D\setminus(D_{1}\cup D_{1}^{\ast}\cup\Omega_{R})}|\nabla v_{1}|^{2}+O(1)\varepsilon\\
=&\int_{D\setminus(D_{1}\cup D_{1}^{\ast}\cup\Omega_{R})}|\nabla v_{1}^{\ast}|^{2}+2\int_{D\setminus(D_{1}\cup D_{1}^{\ast}\cup\Omega_{R})}\nabla v_{1}^{\ast}\cdot\nabla(v_{1}-v_{1}^{\ast})\\
&+\int_{D\setminus(D_{1}\cup D_{1}^{\ast}\cup\Omega_{R})}|\nabla(v_{1}-v_{1}^{\ast})|^{2}+O(1)\varepsilon\\
=&\int_{\Omega^{\ast}\setminus\Omega^{\ast}_{R}}|\nabla v_{1}^{\ast}|^{2}+O(1)\varepsilon^{\frac{1}{6}}.
\end{align*}
This, together with (\ref{con03365}) and (\ref{con037}), yields that
\begin{align*}
a_{11}=&\int_{\varepsilon^{\bar{\gamma}}<|x'|<R}\frac{dx'}{h_{1}(x')-h(x')}+\int_{|x'|<\varepsilon^{\bar{\gamma}}}\frac{dx'}{\varepsilon+h_{1}(x')-h(x')}\\
&+A+O(1)
\begin{cases}
\varepsilon^{\frac{m-1}{6m}},&n=2,\\
\varepsilon^{\frac{1}{6}},&n\geq3.
\end{cases}
\end{align*}

{\bf Step 2.} Denote
$$\mathbf{Main}:=\int_{\varepsilon^{\bar{\gamma}}<|x'|<R}\frac{dx'}{h_{1}(x')-h(x')}+\int_{|x'|<\varepsilon^{\bar{\gamma}}}\frac{dx'}{\varepsilon+h_{1}(x')-h(x')}.$$

(i) For $m\geq n-1$,
\begin{align*}
\mathbf{Main}=&\int_{|x'|<R}\frac{dx'}{\varepsilon+h_{1}-h}+\int_{\varepsilon^{\bar{\gamma}}<|x'|<R}\frac{\varepsilon\,dx'}{(h_{1}-h)(\varepsilon+h_{1}-h)}\\
=&\int_{|x'|<R}\frac{1}{\varepsilon+\lambda|x'|^{m}}+\int_{|x'|<R}\left(\frac{1}{\varepsilon+h_{1}-h}-\frac{1}{\varepsilon+\lambda|x'|^{m}}\right)+O(1)\varepsilon^{\frac{4m+n-1}{6m}}\\
=&(n-1)\omega_{n-1}\int_{0}^{R}\frac{s^{n-2}}{\varepsilon+\lambda s^{m}}+O(1)\int^{R}_{0}\frac{s^{n-1}}{\varepsilon+\lambda s^{m}}\\
=&
\frac{(n-1)\omega_{n-1}\Gamma^{n}_{m}}{m\lambda^{\frac{n-1}{m}}}
\begin{cases}
\varepsilon^{\frac{n-1}{m}-1}+O(1)\varepsilon^{\frac{n}{m}-1},&m>n,\\
\varepsilon^{-\frac{1}{m}}+O(1)|\ln\varepsilon|,&m=n,\\
|\ln\varepsilon|+O(1),&m=n-1;
\end{cases}
\end{align*}

(ii) For $m<n-1$,
\begin{align*}
\mathbf{Main}=&\int_{|x'|<R}\frac{dx'}{h_{1}-h}-\int_{\varepsilon^{\bar{\gamma}}<|x'|<R}\frac{\varepsilon\,dx'}{(h_{1}-h)(\varepsilon+h_{1}-h)}\\
=&\int_{\Omega_{R}^{\ast}}|\partial_{n}\bar{u}^{\ast}|^{2}+O(1)\varepsilon^{\frac{4m+n-1}{6m}}.
\end{align*}

Therefore, it follows from {\bf Step 1} and {\bf Step 2} that Lemma \ref{lem003} holds.

\subsection{Proof of Theorem \ref{thm001}}
Recalling decomposition (\ref{con007}) and combining the results derived in Theorem \ref{thm002}, Lemma \ref{lem001}, Lemma \ref{lem002} and Lemma \ref{lem003}, we complete the proofs of Theorem \ref{thm001} and Theorem \ref{coro002}.

\end{proof}

\section{Appendix:\,The proof of Theorem \ref{thm002}}

In light of assumptions ({\bf{\em H1}}) and ({\bf{\em H2}}), it follows from a direct calculation that for $i=1,\cdots,n-1$, $x\in\Omega_{2R}$,
\begin{align}
|\partial_{i}\bar{v}|\leq&\frac{C|\psi(x',\varepsilon+h_{1}(x'))|}{\sqrt[m]{\varepsilon+|x'|^{m}}}+C\|\nabla\psi\|_{L^{\infty}(\partial D_{1})},\label{ADE001}\\
|\partial_{n}\bar{v}|=&\frac{|\psi(x',\varepsilon+h_{1}(x'))|}{\delta(x')},\quad\partial_{nn}\bar{v}=0,\label{ADE002}
\end{align}
and
\begin{align}\label{ADE003}
|\Delta\bar{v}|\leq\frac{|\psi(x',\varepsilon+h_{1}(x'))|}{(\varepsilon+|x'|^{m})^{\frac{2}{m}}}+\frac{\|\nabla\psi\|_{L^{\infty}(\partial D_{1})}}{\sqrt[m]{\varepsilon+|x'|^{m}}}+\|\nabla^{2}\psi\|_{L^{\infty}(\partial D_{1})}.
\end{align}
Here and throughout this section, for simplicity of notations, we use $\|\nabla\psi\|_{L^{\infty}}$, $\|\nabla^{2}\psi\|_{L^{\infty}}$ and $\|\psi\|_{C^{2}}$ to denote $\|\nabla\psi\|_{L^{\infty}(\partial{D}_{1})}$, $\|\nabla^{2}\psi\|_{L^{\infty}(\partial{D_1})}$ and $\|\psi\|_{C^{2}(\partial D_{1})}$, respectively.

Define
\begin{equation}\label{def_w}
w:=v-\bar{v}.
\end{equation}

\noindent{\bf STEP 1.}
Let $v\in H^1(\Omega)$ be a weak solution of (\ref{con008}). Then
\begin{align}\label{ADE005}
\int_{\Omega}|\nabla w|^2dx\leq C\|\psi\|_{C^{2}(\partial D_{1})}^{2}.
\end{align}

Invoking (\ref{def_w}), $w$ satisfies
\begin{align}\label{ADE0006}
\begin{cases}
\Delta w=-\Delta\bar{v},&
\hbox{in}\  \Omega,  \\
w=0, \quad&\hbox{on} \ \partial\Omega.
\end{cases}
\end{align}
Multiplying the equation in (\ref{ADE0006}) and integrating by parts, it follows from the Poincar\'{e} inequality, Sobolev trace embedding theorem, (\ref{KK6}) and (\ref{ADE001})--(\ref{ADE002}) that
\begin{align*}
\int_{\Omega}|\nabla w|^{2}=&\int_{\Omega_{R}}\omega\Delta\bar{v}+\int_{\Omega\setminus\Omega_{R}}\omega\Delta\bar{v}\\
\leq&\sum^{n-1}_{i=1}\left|\int_{\Omega_{R}}\omega\partial_{ii}\bar{v}\right|+C\|\psi\|_{C^{2}}\int_{\Omega\setminus\Omega_{R}}|w|\\
\leq&C\|\nabla w\|_{L^{2}(\Omega_{R})}\|\nabla_{x'}\bar{v}\|_{L^{2}(\Omega_{R})}+C\|\psi\|_{C^{2}(\partial D_{1})}\|\nabla w\|_{L^{2}(\Omega\setminus\Omega_{R})}\\
\leq&C\|\psi\|_{C^{2}(\partial D_{1})}\|\nabla w\|_{L^{2}(\Omega)}.
\end{align*}
Then (\ref{ADE005}) is proved.

\noindent{\bf STEP 2.}
Proof of
\begin{align}\label{step2}
 \int_{\Omega_\delta(z')}|\nabla w|^2dx
 &\leq C\delta^{n+2-\frac{4}{m}}\left(|\psi(z',\varepsilon+h_{1}(z'))|^2+\delta^{\frac{2}{m}}\|\psi\|_{C^2(\partial D_1)}^2\right),
\end{align}
where $\delta$ is defined by (\ref{ZZW666}). As seen in \cite{LX2017}, we have the iteration formula as follows:
\begin{align*}
\int_{\Omega_{t}(z')}|\nabla w|^{2}dx\leq\frac{C}{(s-t)^{2}}\int_{\Omega_{s}(z')}|w|^{2}dx+C(s-t)^{2}\int_{\Omega_{s}(z')}|\Delta\bar{v}|^{2}dx.
\end{align*}

We next divide into two cases to prove (\ref{step2}).

{\bf Case 1.} If $|z'|<\varepsilon^{\frac{1}{m}},\,0<s<\varepsilon^{\frac{1}{m}}$, we have $\varepsilon\leq\delta(x')\leq C\varepsilon$ in $\Omega_{\sqrt[m]{\varepsilon}}(z')$. In light of (\ref{ADE003}), we derive
\begin{align}\label{ADE006}
&\int_{\Omega_{s}(z')}|\Delta\bar{v}|^{2}\leq C|\psi(z',\varepsilon+h_{1}(z'))|^{2}\frac{s^{n-1}}{\varepsilon^{\frac{4}{m}-1}}+Cs^{n-1}\varepsilon^{1-\frac{2}{m}}\|\psi\|^{2}_{C^{2}},
\end{align}
while, due to the fact that $w=0$ on $\Gamma^{-}_{R}:=\{x\in\mathbb{R}^{n}|\,x_{n}=h(x'),\,|x'|<R\}$,
\begin{align}\label{ADE007}
\int_{\Omega_{s}(z')}|w|^{2}\leq C\varepsilon^{2}\int_{\Omega_{s}(z')}|\nabla w|^{2}.
\end{align}
Denote
$$F(t):=\int_{\Omega_{t}(z')}|\nabla w_{1}|^{2}.$$
It follows from (\ref{ADE006}) and (\ref{ADE007}) that for $0<t<s<\varepsilon^{\frac{1}{m}}$,
\begin{align}\label{ADE008}
F(t)\leq &\left(\frac{c_1\varepsilon}{s-t}\right)^2F(s)+C(s-t)^2s^{n-1}\bigg(\frac{|\psi(z',\varepsilon+h_{1}(z'))|^2}{\varepsilon^{\frac{4-m}{m}}}+\frac{\|\psi\|_{C^2}}{\varepsilon^{\frac{2-m}{m}}}\bigg),
\end{align}
where $c_{1}$ and $C$ are universal constants.

Pick $k=\left[\frac{1}{4c_{1}\sqrt[m]{\varepsilon}}\right]+1$ and $t_{i}=\delta+2c_{1}i\varepsilon,\;i=0,1,2,\cdots,k$. Then, (\ref{ADE008}), together with $s=t_{i+1}$ and $t=t_{i}$, leads to
$$F(t_{i})\leq\frac{1}{4}F(t_{i+1})+C(i+1)^{n-1}\varepsilon^{n+2-\frac{4}{m}}\left[|\psi(z',\varepsilon+h_{1}(z'))|^{2}+\varepsilon^{\frac{2}{m}}\|\psi\|^{2}_{C^{2}}\right].$$
It follows from $k$ iterations and (\ref{ADE005}) that for a sufficiently small $\varepsilon>0$,
\begin{align}\label{ADE009}
F(t_{0})\leq C\varepsilon^{n+2-\frac{4}{m}}\left(|\psi(z',\varepsilon+h_{1}(z'))|^{2}+\varepsilon^{\frac{2}{m}}\|\psi\|^{2}_{C^{2}}\right).
\end{align}

{\bf Case 2.} If $\varepsilon^{\frac{1}{m}}\leq|z'|\leq R$ and $0<s<\frac{2|z'|}{3}$, we have $\frac{|z'|^{m}}{C}\leq\delta(x')\leq C|z'|^{m}$ in $\Omega_{\frac{2|z'|}{3}}(z')$. Similar to (\ref{ADE006}) and (\ref{ADE007}), we obtain
\begin{align*}
\int_{\Omega_{s}(z')}|\Delta\bar{v}|^{2}\leq&C|\psi(z',\varepsilon+h_{1}(z'))|^{2}\frac{s^{n-1}}{|z'|^{4-m}}+Cs^{n-1}|z'|^{m-2}\|\psi\|^{2}_{C^{2}},
\end{align*}
and
$$\int_{\Omega_{s}(z')}|w|^{2}\leq C|z'|^{2m}\int_{\Omega_{s}(z')}|\nabla w|^{2}.$$
Moreover, for $0<t<s<\frac{2|z'|}{3}$, estimate (\ref{ADE008}) becomes
\begin{align*}
F(t)\leq\left(\frac{c_{2}|z'|^{m}}{s-t}\right)^{2}F(s)+C(s-t)^{2}s^{n-1}\left(\frac{|\psi(z',\varepsilon+h_{1}(z'))|^{2}}{|z'|^{4-m}}+|z'|^{m-2}\|\psi\|_{C^{2}}^{2}\right).
\end{align*}

Similarly as above, pick $k=\left[\frac{1}{4c_{2}|z'|}\right]+1,\,t_{i}=\delta+2c_{2}i|z'|^{m},\,i=0,1,2,\cdots,k$ and take $s=t_{i+1},\;t=t_{i}$. Then, we obtain
$$F(t_{i})\leq\frac{1}{4}F(t_{i+1})+C(i+1)^{n-1}|z'|^{m(n+2)-4}\left(|\psi(z',\varepsilon+h_{1}(z'))|^{2}+|z'|^{2}\|\psi\|^{2}_{C^{2}}\right).$$
Likewise, by using $k$ iterations, we have
\begin{align}\label{ADE010}
F(t_{0})\leq C|z'|^{m(n+2)-4}\left(|\psi(z',\varepsilon+h_{1}(z'))|^{2}+|z'|^{2}\|\psi\|^{2}_{C^{2}}\right).
\end{align}
Consequently, (\ref{ADE010}), together with (\ref{ADE009}), yields that (\ref{step2}) holds.

\noindent{\bf STEP 3.}
Proof of
\begin{align}\label{ADE011}
|\nabla w(x)|\leq C\delta^{1-\frac{2}{m}}(|\psi(x',\varepsilon+h_{1}(x'))|+\delta^{\frac{1}{m}}\|\psi\|_{C^{2}(\partial D_{1})}),\quad\mathrm{in}\;\Omega_{R}.
\end{align}

As in \cite{LX2017}, combining the rescaling argument, Sobolev embedding theorem, $W^{2,p}$ estimate and bootstrap argument, we obtain
\begin{align*}
\|\nabla w\|_{L^{\infty}(\Omega_{\delta/2}(z'))}\leq\frac{C}{\delta}\left(\delta^{1-\frac{n}{2}}\|\nabla w\|_{L^{2}(\Omega_{\delta}(z'))}+\delta^{2}\|\Delta\bar{v}\|_{L^{\infty}(\Omega_{\delta}(z'))}\right).
\end{align*}
In view of (\ref{ADE003}) and (\ref{step2}), we obtain that for $|z'|\leq R$,
$$\delta\|\Delta\bar{v}\|_{L^{\infty}(\Omega_{\delta}(z'))}\leq C\delta^{1-\frac{2}{m}}(|\psi(z',\varepsilon+h_{1}(z'))|+\delta^{\frac{1}{m}}\|\psi\|_{C^{2}}),$$
and
\begin{align*}
\delta^{-\frac{n}{2}}\|\nabla w\|_{L^{2}(\Omega_{\delta}(z'))}\leq C\delta^{1-\frac{2}{m}}(|\psi(z',\varepsilon+h_{1}(z'))|+\delta^{\frac{2}{m}}\|\psi\|_{C^{2}}).
\end{align*}
Consequently, for $h(z')<z_{n}<\varepsilon+h_{1}(z')$,
$$|\nabla w(z',z_{n})|\leq C\delta^{1-\frac{2}{m}}(|\psi(z',\varepsilon+h_{1}(z'))|+\delta^{\frac{1}{m}}\|\psi\|_{C^{2}}).$$
Estimate (\ref{con013}) is established. On the other hand, it follows from the standard interior estimates and boundary estimates for the Laplace equation that
\begin{align*}
\|\nabla v\|_{L^{\infty}(\Omega\setminus\Omega_{R})}\leq C\|\psi\|_{C^{2}(\partial D_{1})}.
\end{align*}
Thus, Theorem \ref{thm002} is proved.

\noindent{\bf{\large Acknowledgements.}} The author is greatly indebted to Professor HaiGang Li for his constant encouragement and very helpful discussions. The author was partially supported by NSFC (11971061) and BJNSF (1202013).

\bibliographystyle{plain}

\def\cprime{$'$}

\end{document}